\documentclass[english,10pt]{amsart}
\usepackage{amsmath, amssymb, amsthm, enumerate}
\usepackage[all]{xy}
\usepackage[activeacute]{babel}
\usepackage[alphabetic]{amsrefs}
\newtheorem{thm}[equation]{Theorem}
\makeatletter
\let\c@subsubsection\c@equation
\makeatother

\newtheorem{prop}[equation]{Proposition}

\theoremstyle{remark}
\newtheorem{rmk}[equation]{Remark}

\theoremstyle{definition}

\newcommand{\Hom}{\mathrm{Hom}}
\newcommand{\inthomeff}{\mathbf{hom}^{\mathrm{eff}}}

\newcommand{\spec}[1]{\mathrm{Spec}(#1)}

\newcommand{\sphere}{\mathbf 1}

\newcommand{\DMeff}[1]{DM^{\mathrm{eff}}_{#1}}
\newcommand{\DM}[1]{DM_{#1}}
\newcommand{\northogonal}[2]{DM _{#1}^{\perp}(#2)}
\newcommand{\HINST}[1]{\mathbf{HI}_{#1}}

\numberwithin{equation}{subsection}

\begin{document}

%%%%%%%%%%%%%%%%%%%%%%%%%%%%%%%%%%%%%%%%%%%%
%%%%%%%%%%%%%%%%%%%%%%%%%%%%%%%%%%%%%%%%%%%%
%%%%%%%%%%%%%%%%%%%%%%%%%%%%%%%%%%%%%%%%%%%%

\title{Some remarks on Chow correspondences}

%\dedicatory{}

%%%%%%%%%%%%%%%%%%%%%%%%%%%%%%%%%%%%%%%%%%%%
%%%%%%%%%%%%%%%%%%%%%%%%%%%%%%%%%%%%%%%%%%%%
%%%%%%%%%%%%%%%%%%%%%%%%%%%%%%%%%%%%%%%%%%%%

\author{Pablo Pelaez}
\address{Chebyshev Laboratory\\
St. Petersburg State University\\
14th Line V. O., 29B\\
Saint Petersburg 199178 Russia}
\email{pablo.pelaez@gmail.com}

%%%%%%%%%%%%%%%%%%%%%%%%%%%%%%%%%%%%%%%%%%%%
%%%%%%%%%%%%%%%%%%%%%%%%%%%%%%%%%%%%%%%%%%%%
%%%%%%%%%%%%%%%%%%%%%%%%%%%%%%%%%%%%%%%%%%%%



%%%%%%%%%%%%%%%%%%%%%%%%%%%%%%%%%%%%%%%%%%%%
%%%%%%%%%%%%%%%%%%%%%%%%%%%%%%%%%%%%%%%%%%%%
%%%%%%%%%%%%%%%%%%%%%%%%%%%%%%%%%%%%%%%%%%%%

\begin{abstract}
We study, in the context of Voevodsky's triangulated category
of motives, several adequate equivalence relations (in the sense of
Samuel) on the graded  Chow ring $CH^\ast 
(X\times Y)$ for $X$, $Y$ smooth projective varieties over a field.
\end{abstract}

%%%%%%%%%%%%%%%%%%%%%%%%%%%%%%%%%%%%%%%%%%%%%%
%%%%%%%%%%%%%%%%%%%%%%%%%%%%%%%%%%%%%%%%%%%%%%
%%%%%%%%%%%%%%%%%%%%%%%%%%%%%%%%%%%%%%%%%%%%%%
%\thanks{The author was supported by the Russian Science
%Foundation grant 19-71-30002.}
 
\maketitle

%%%%%%%%%%%%%%%%%%%%%%%%%%%%%%%%%%%%%%%%%%%%%%
%%%%%%%%%%%%%%%%%%%%%%%%%%%%%%%%%%%%%%%%%%%%%%
%%%%%%%%%%%%%%%%%%%%%%%%%%%%%%%%%%%%%%%%%%%%%%
\section{Introduction}  \label{sec.introd}

\subsection{}
Let $k$ be an algebraically closed base field, and $X$ a smooth projective
variety over $k$ of dimension $d$.  Let $CH^m(X)$ be the group
of codimension $m$ algebraic cycles of $X$ modulo rational equivalence, 
$1\leq m\leq d$, and $CH^m_{\mathrm{alg}}(X)\subseteq CH^m(X)$ be
the subgroup of algebraic cycles which are algebraically equivalent to zero.
We will write $\mathrm{alb}_X: CH^d_{\mathrm{alg}}(X)\rightarrow 
\mathrm{Alb}(X)(k)$ for the Albanese map.  We fix a $k$-rational point
$x_0\in X(k)$, and consider a canonical map $f:X\rightarrow
\mathrm{Alb}(X)$ such that $f(x_0)=0\in \mathrm{Alb}(X)(k)$.

Let  $\Lambda \in CH^{d+n}(X\times Y)$ be a Chow corrrespondence,
where $Y$ is a smooth projective variety of dimension $d_Y$, $d_Y\geq n\geq0$, and consider the following homomorphism of abelian groups: 
\begin{align*}
\psi : & CH^{d_Y -n}(Y) \rightarrow
\mathrm{Alb}(X)(k)\\
& \gamma  \mapsto \mathrm{alb}_X(\Lambda (\gamma)
-\mathrm{deg}(\Lambda(\gamma))[x_0]),
\end{align*}
with  $\Lambda (\gamma)=p_{X\ast}(p_Y^\ast(\gamma)\cdot \Lambda)\in CH^d(X)$, where
$p_X:Y\times X \rightarrow X$, $p_Y:Y\times X \rightarrow Y$
are the projections.  

We  observe that $\psi$
restricted to
$CH^{d_Y -n}_{\mathrm{alg}}(Y)$
is a cycle induced regular 
homomorphism in the sense of Lieberman \cite{MR238857}*{p. 1194},
\eqref{def.reg.hom}.
Noteworthy examples are, in the case $k=\mathbb C$ and
$A$ an abelian variety, Weil's \cite{MR0050330}*{\S IV.27}
Abel-Jacobi maps $\psi_n : CH^n_{\mathrm{alg}}(A)\rightarrow
J^n_a(A)(\mathbb C)$ \cite{MR238857}*{Prop. 6.7, Thm. 6.5}
into the algebraic part  of the Weil-Griffiths
intermediate Jacobian $J^n(A)$  \cite{MR0050330}*{\S IV.24-26},
\cite{MR0233825}*{Ex. 2.1, Thm. 2.54}.

Our  goal is to characterize, in the context of Voevodsky's
triangulated category of motives $\DM{k}$, those correspondences
$\Lambda \in CH^{d+n}(X\times Y)$ for which the homomorphism
$\psi$ is zero.  Namely, we give a criterion, in terms of the
functor of points $\mathcal A \in \DM{k}$ of the Albanese variety
$\mathrm{Alb}(X)$, and its orthogonal cover 
$\theta ^{\mathcal A}_{-1}:
bc_{\leq -1} (\mathcal A)\rightarrow \mathcal A$ in $\DM{k}$
\cite{MR3614974}, \cite{MR4486247}, which determines those
correspondences $\Lambda$ for which the homomorphism $\psi$
is identically zero.  Furthermore, we also show that the orthogonal
cover $\theta ^{\mathcal A}_{-1}$ detects those correspondences
$\Lambda$ which are in $CH^{d+n}_{\mathrm{alg}^{\ast 2}}(X\times Y)$
(square equivalent to zero \eqref{sss.prod-eqrel})
or in $CH^{d+n}_{H^{\ast 2}}(X\times Y)$ (the second step of H. Saito's
filtration \eqref{sss.HSai-fil}).

The paper is organized as follows:  in section \ref{sec.preel},
 we establish the notation
and collect some results that will be used in the rest of the paper.
In section \ref{s.ad-eqrel}, we discuss, in the context
of Voevodsky's triangulated category of motives,
adequate equivalence relations
on algebraic cycles on smooth projective varieties.  In section 
\ref{s.mainres}, we prove our main result \eqref{thm.main}
and show that the orthogonal filtration detects
correspondences which are square equivalent to zero
\eqref{thm.square-eq} or which are in $H^{\ast 2}$,
the second step of
H. Saito's filtration, \eqref{thm.square-homeq}.

\section{Preliminaries}  \label{sec.preel}

In this section we fix the notation that will be used in the
rest of the paper and recall some results from the literature
which will be necessary in the sequel.  The results of this
section are not original.

\subsection{Definitions and Notation}	\label{subsec.defandnots}		

We consider an arbitrary base field $k$, and let $Sch_k$ be the category
of $k$-schemes of finite type.  
If $X$, $Y\in Sch_k$, we will write $X\times Y$ for
$X\times _{\spec k}Y$.
Let
$\pi _X :X\rightarrow \spec{k}$ be the structure map of $X\in Sch_k$,
and $X(k)$  its set of $k$-points. 
Given
$x\in X(k)$, we will  also write 
$x:\spec{k}\rightarrow X$ for the corresponding map in $Sch_k$.

Let $Sm_k$ be the full
subcategory of $Sch_k$ which consists of smooth $k$-schemes
considered as a site with the Nisnevich topology.  
We will write $SmProj_k$ 
for the full subcategory of $Sm_k$ consisting of  smooth projective
varieties.
When $X\in SmProj_k$ of dimension $d$ and
 $x\in X(k)$, let $[x] \in CH^d(X)$ be the
zero-cycle associated to $x$ as a closed subscheme of $X$.  
Given $\alpha \in CH^d(X)$, we will write
$\mathrm{deg} (\alpha)\in \mathbb Z$ for $\pi _{X\ast}(\alpha)\in
CH^0({\spec{k}})\cong \mathbb Z$.

We will use freely the language of triangulated categories
\cite{MR1812507}, \cite{MR751966}.  
We will write $[1]$ (resp. $[-1]$) for the suspension (resp. 
desuspension) functor in a triangulated category; and for $n>0$, the
composition of $[1]$ (resp. $[-1]$) iterated $n$-times will be $[n]$
(resp. $[-n]$).  If $n=0$, then $[0]$ will be the identity functor.

In all the categories under consideration, $0$ will be the zero
object (if it exists), and $\cong$ will denote that a map (resp. a functor)
is an isomorphism (resp. an equivalence of categories).

\subsection{Voevodsky's triangulated category of motives}  \label{ss.DMdef}

Let $Cor_k$ be the Suslin-Voevodsky category of finite
correspondence over $k$ \cite{MR1744945}, \cite{MR3590347},
and $Shv^{tr}$ be the category of Nisnevich sheaves with
transfers (which is an abelian category \cite{MR2242284}*{13.1}).
Given $X\in Sm_k$, we will write
 $\mathbb Z _{tr}(X)$ for the Nisnevich
sheaf with transfers represented by $X$ \cite{MR2242284}*{2.8 and 6.2}.

We will write 
$K(Shv ^{tr})$ for the category of unbounded chain complexes
on $Shv^{tr}$ equipped with the injective model structure
\cite{MR1780498}*{Prop. 3.13}, and  $D(Shv^{tr})$ for its
homotopy category.  Let $K^{\mathbb A ^1}(Shv ^{tr})$ be
the left Bousfield localization \cite{MR1944041}*{3.3} of $K(Shv^{tr})$
with respect to the set of maps 
$\{ \mathbb Z _{tr}(X\times _k \mathbb A ^1)[n] \rightarrow
\mathbb Z _{tr}(X)[n] : X\in Sm_k; n\in \mathbb Z \}$ induced by the
projections $p: X\times _k \mathbb A ^1 \rightarrow X$.
Voevodsky's triangulated category of effective motives $\DMeff{k}$
is the homotopy category of $K^{\mathbb A ^1}(Shv ^{tr})$
\cite{MR1764202}.

Let $T\in K^{\mathbb A ^1}(Shv ^{tr})$ be the
chain complex $\mathbb Z _{tr}(\mathbb G _m)[1]$ 
\cite{MR2242284}*{2.12}, where $\mathbb G _m$ is the $k$-scheme
$\mathbb A ^1 \backslash \{ 0 \}$ pointed by $1$.  Consider the category
$Spt _T (Shv ^{tr})$ of symmetric $T$-spectra on
$ K^{\mathbb A ^1}(Shv ^{tr})$ equipped with the model
structure defined in \cite{MR1860878}*{8.7 and 8.11},
\cite{MR2438151}*{Def. 4.3.29}.  Voevodsky's triangulated
category of motives $\DM{k}$ is the homotopy category of
$Spt _T (Shv ^{tr})$ \cite{MR1764202}.

If $X\in Sm_k$, let $M(X)$ be the image of 
$\mathbb Z _{tr}(X) \in D(Shv ^{tr})$ under the $\mathbb A ^1$-localization
map $D(Shv ^{tr})\rightarrow \DMeff{k}$.  We will write
$\Sigma ^\infty : \DMeff{k}\rightarrow \DM{k}$ for the suspension functor
\cite{MR1860878}*{7.3}, we will abuse notation and simply write $E$ for
$\Sigma ^\infty E$, $E\in \DMeff{k}$.  

We recall that $\DMeff{k}$ and $\DM{k}$ are tensor
triangulated categories \cite{MR2438151}*{Thm. 4.3.76 and Prop. 4.3.77}
with unit $\sphere = M(\spec{k})$.  Let $E(1)$ be
$E\otimes M(\mathbb G _m)[-1]$, $E\in \DM{k}$ and inductively
$E(n)=(E(n-1))(1)$, $n\geq 0$.  We observe that the functor
$\DM{k}\rightarrow \DM{k}$, $E\mapsto E(1)$ is an equivalence of
categories \cite{MR1860878}*{8.10}, \cite{MR2438151}*{Thm. 4.3.38};
let $E\mapsto E (-1)$ be its inverse, and inductively
$E(-n)=(E(-n+1))(-1)$, $n>0$.  By convention $E(0)=E$ for
$E\in \DM{k}$.

\subsection{Voevodsky's slice filtration}  \label{ss.sf}
We recall that $\DM{k}$ is a compactly generated triangulated
category \cite{MR1308405}*{Def. 1.7} with the following set of
compact generators \cite{MR2438151}*{Thm. 4.5.67}:

\begin{align}  \label{eq.DMgens}
\mathcal G _{\DM{k}} = \{ M (X)(r) : X\in Sm_k; r\in \mathbb Z \} .
\end{align}

Let $m\in \mathbb Z$, and consider:

\begin{align}  \label{eq.DMeffgenstw}
\mathcal G ^{\mathrm{eff}}(m) = \{ M(X)(r): X\in Sm_k; r\geq m \}
\subseteq \mathcal G _{\DM{k}}.
\end{align}
\subsubsection{}  \label{def.slicefil}
We define $\DMeff{k} (m)$ to be the smallest full triangulated
subcategory of $\DM{k}$ which contains
 $\mathcal G ^{\mathrm{eff}}(m)$
\eqref{eq.DMeffgenstw} and is closed under arbitrary (infinite)
coproducts.
The slice filtration \cite{MR1977582}, \cite{MR2600283}*{p. 18},
\cite{MR2249535} is the following tower of triangulated subcategories
of $\DM{k}$:

\begin{align}  \label{eq.slice.filtration}
\cdots \subseteq \DMeff{k} (m+1) \subseteq
\DMeff{k} (m) \subseteq \DMeff{k} (m-1)
\subseteq \cdots
\end{align}
We notice that the
inclusion
$i_m:\DMeff{k} (m)\rightarrow \DM{k}$ admits a right adjoint
$r_m : \DM{k}\rightarrow \DMeff{k} (m)$ which is also a triangulated
functor \cite{MR1308405}*{Thm. 4.1}.  

The
$(m-1)$-effective cover of the slice filtration is defined to be
$f_m = i_m \circ r_m : \DM{k}\rightarrow \DM{k}$ 
\cite{MR1977582}, \cite{MR2600283}*{p. 18},
\cite{MR2249535}.  We will write $\epsilon _m: f_m \rightarrow id$,
for the counit of the adjunction $(i_m, r_m)$.

We recall that the
suspension functor $\Sigma ^\infty : \DMeff{k} \rightarrow
\DM{k}$ induces an equivalence of categories between $\DMeff{k}$
and $\DMeff{k} (0)$ \cite{MR2804268}*{Cor. 4.10} (in case $k$
is non-perfect of characteristic $p$, it is necessary to consider 
$\mathbb Z [\frac{1}{p}]$-coefficients \cite{MR3590347}*{Cor. 4.13, Thm.
4.12 and Thm. 5.1}).

\subsection{The orthogonal filtration}
Let $m\in \mathbb Z$.  We will write $\northogonal{k}{m}$ for the full
subcategory of $\DM{k}$ consisting of objects $E\in \DM{k}$ such that
for every $G\in \DMeff{k} (m)$ \eqref{def.slicefil}:
$\Hom _{\DM{k}} (G,E)=0$.

\subsubsection{}
The orthogonal filtration \cite{MR3614974}*{3.2.1}
(called birational in \cite{MR3614974}*{3.2}) consists of
the following tower of  triangulated subcategories of $DM_k$:
\begin{align}  \label{eq.orthogonal.filtration}
\cdots \subseteq \northogonal{k}{m-1}\subseteq \northogonal{k}{m} \subseteq 
\northogonal{k}{m+1} \subseteq \cdots
\end{align}

\subsubsection{}  \label{def.orthogonal.adj}
We recall 
that the  inclusion $j_m:\northogonal{k}{m}\rightarrow DM_k$ admits
a right adjoint $p_m: DM_k \rightarrow \northogonal{k}{m}$ which is also
a triangulated functor \cite{MR1308405}*{Thm. 4.1},
\cite{MR3614974}*{3.2.2}.  Let $bc_{\leq m}=j_{m+1}\circ p_{m+1}:DM_k \rightarrow
DM_k$,
where $bc$ stands for birational cover \cite{MR3614974}*{3.2}.
The counit of the adjunction, $\theta _m :bc_{\leq m} \rightarrow id$,
 satisfies a universal property \cite{MR3614974}*{3.2.4},
which combined with the inclusions \eqref{eq.orthogonal.filtration} induces
a canonical natural transformation $bc_{\leq m}\rightarrow
bc_{\leq m+1}$.

\subsection{}  \label{ss.nt.htstr}
We will consider Voevodsky's homotopy $t$-structure
$((\DMeff{k})_{\geq 0}, (\DMeff{k})_{\leq 0})$ in $\DMeff{k}$
\cite{MR1764202}*{p. 11}.  We will use the homological notation
for $t$-structures \cite{MR2438151}*{\S 2.1.3}, \cite{MR2735752}*{\S 1.3},
and write $\tau _{\geq m}$, $\tau _{\leq m}$ for the truncation
functors and $\mathbf{h}_m = [-m](\tau _{\leq m} \circ \tau _{\geq m})$.
We will write  $\HINST{k}$ for the abelian category of homotopy invariant
Nisnevich sheaves with transfers on $Sm_k$, which is the heart
of the homotopy $t$-structure in $\DMeff{k}$.

\subsubsection{}  \label{ss.nt.abvar}
Consider an abelian variety
$A\in SmProj_k$.  We will write
$\mathcal A \in \HINST{k}$ for the functors of points of $A$ considered
as a homotopy invariant Nisnevich sheaf with transfers
\cite{MR1957261}*{Lem. 3.2, Rmk. 3.3}.  Given a map
$f:Y\rightarrow A$ in $Sm_k$,  we will  write
$\mathcal A
_f :M(Y)\rightarrow \mathcal A$ in $\DMeff{k}$ for the image of $f$
under the isomorphism
$\Gamma (Y, \mathcal A)\cong  \Hom_{\DMeff{k}}(M(Y), \mathcal A)$
\cite{MR1764202}*{Prop. 3.1.9 and 3.2.3}.

\subsubsection{}  \label{prop.kpts.sep}
With the notation and conditions of \eqref{ss.nt.abvar}.  Assume
that the base field $k$ is algebraically closed, let
$\mathcal F \in \HINST{k}$ \eqref{ss.nt.htstr}, and consider a map
$f:\mathcal F \rightarrow \mathcal A$ in $\HINST{k}$.  
The following is well known \cite{2025arXiv250915920H}*{2.4.2}:
$f=0$ if and only if the induced map
$f_k =0: \Gamma (\spec{k}, \mathcal F)
\rightarrow \Gamma (\spec{k}, \mathcal A)$.

\section{Adequate equivalence relations and Voevodsky's motives}
\label{s.ad-eqrel}

\subsection{}  \label{not.act.cor}
Let  $X$, $Y\in SmProj_k$ of dimension $d$, $d_Y$, respectively,
and $\Lambda \in CH^l (Y\times X)$.
Given $\gamma \in CH^{m}(Y)$,
we will write $\Lambda (\gamma)\in CH^n(X)$, $n=l+m-d_Y$, for
$p_{X\ast}(p_Y^\ast(\gamma)\cdot \Lambda)$ where
$p_X:Y\times X \rightarrow X$, $p_Y:Y\times X \rightarrow Y$
are the projections.

We will consider the isomorphism \cite{MR1883180}, $Z\in Sm_k$:
\begin{align}  \label{eq.Voev.comp}
\Hom _{\DM{k}}(M(Z), \sphere (m)[2m])\cong CH^m(Z),
\end{align}
as an identification.
\subsubsection{}  \label{sss.ident.chgps}
Given
$\alpha \in CH^m(Z)$, $Z\in Sm_k$, we will
also write $\alpha: M(Z)\rightarrow \sphere (m)[2m]$ for the
map in $\DM{k}$ corresponding to $\alpha$ under \eqref{eq.Voev.comp}.

\subsubsection{}   \label{sss.dual.corr}
With the notation and conditions of \eqref{not.act.cor}.  
Consider $\Lambda : M(Y)\otimes M(X)
\rightarrow \sphere (l)[2l]$ in $\DM{k}$, \eqref{sss.ident.chgps}.
Dualizing $M(Y)$ \cite{MR1764202}*{Thm. 4.3.7},
\cite{MR2399083}*{Prop. 6.7.1 and \S 6.7.3}  we obtain the
following map in $\DM{k}$:
\begin{align*}
\xymatrix{\Lambda _Y : M(X) \ar[r]& M(Y) (-d_Y+l)[-2d_Y+2l]}
\end{align*}

\subsection{Adequate equivalence relations}  \label{ss.ad.eq.rel}
An adequate equivalence relation E on algebraic cycles on
$SmProj_k$ \cite{MR116010}*{\S 1},
consists  of subgroups \cite{MR1189891}*{Prop. 1.2.1},
\cite{MR1744947}*{Lem 1.3}:
\begin{align*}
  \{ CH^n_E(X)\subseteq CH^n(X): X\in SmProj_k, n\geq 0 \}
\end{align*}
such that for every $\Lambda \in CH^l (Y\times X)$;
$X$, $Y\in SmProj_k$, and every $\gamma \in CH^m_E(Y)$:
$\Lambda (\gamma) \in CH^n_E(X)$ \eqref{not.act.cor}.

\subsubsection{}   \label{sss.commacat.equivrel}
With the notation and conditions of \eqref{not.act.cor} and
\eqref{eq.Voev.comp}-\eqref{sss.dual.corr}.
Consider a map $h: M \rightarrow \sphere$ in $\DM{k}$.
Then we obtain an adequate equivalence relation $E_h$
\eqref{ss.ad.eq.rel} by defining $CH^n_{E_h}(X)$ as the
image of the map induced by $h$:
\begin{align*}
h_\ast: \Hom _{\DM{k}}(M(X)(-n)[-2n], M) \rightarrow
\Hom _{\DM{k}}(M(X)(-n)[-2n], \sphere) \overset{1}{\cong}
CH^n(X)
\end{align*}
where the isomorphism (1) follows by adjointness and \eqref{eq.Voev.comp}, \cite{MR1883180}.

In effect, let $\gamma \in CH^m_{E_h}(Y)$, $\Lambda \in CH^l (Y\times X)$,
$n=l+m-d_Y$, and consider $\beta = \Lambda _Y (-n)[-2n]: M(X)(-n)[-2n]
\rightarrow M(Y)(-m)[-2m]$ in $\DM{k}$ \eqref{sss.dual.corr}.  Thus,
$\gamma (-m)[-2m]: M(Y)(-m)[-2m]
\rightarrow \sphere$ \eqref{sss.ident.chgps} is in the image of $h_\ast$, so,
$\gamma (-m)[-2m]\circ \beta: M(X)(-n)[-2n]\rightarrow \sphere$ is also
in the image of $h_\ast$.  Then, by
combining Lieberman's lemma \cite{MR1644323}*{Prop. 16.1.1} and
\cite{MR1764202}*{Prop. 2.1.4, Thm. 3.2.6} we conclude that
under the isomorphism $\Hom _{\DM{k}}(M(X)(-n)[-2n], \sphere)
\cong CH^n(X)$, $\Lambda (\gamma)$ \eqref{not.act.cor} corresponds
to $\gamma (-m)[-2m]\circ \beta$, and the assertion follows. 

\subsubsection{}
Conversely, given an adequate equivalence relation 
$E$ \eqref{ss.ad.eq.rel} on algebraic cycles on $SmProj_k$,
there exists a (non-unique) map $h:M\rightarrow \sphere$ in $\DM{k}$ such that $E_h=E$ \eqref{sss.commacat.equivrel}.  

In effect, let:
\begin{align*}
	M=\bigoplus _{X\in SmProj_k} \bigoplus _{\alpha \in CH^n_E(X), 
	n\geq 0} 
	M(X)(-n)[-2n]
\end{align*}
and consider the map $h:M\rightarrow \sphere$ in $\DM{k}$ induced by
$\alpha (-n)[-2n]:M(X)(-n)[-2n] \rightarrow \sphere ,$ 
\eqref{sss.ident.chgps} on each
direct sumand of $M$.

\subsubsection{}
Consider the map in $\DM{k}$ given by
the counit $\theta _{-1}$ \eqref{def.orthogonal.adj} in the orthogonal filtration:
\begin{align*}
 \theta _{-1}^\sphere: bc_{\leq -1}(\sphere ) \rightarrow \sphere
\end{align*}
Then, it follows from \cite{MR3614974}*{5.3.6}, \cite{pelaez2025incidenceequivalenceblochbeilinsonfiltration}*{2.5.12} that
the equivalence relation $E_{\theta _{-1}^\sphere}$ 
\eqref{sss.commacat.equivrel} on algebraic cycles on $SmProj_k$ is
 numerical equivalence \cite{MR116010}*{2.3}.
 
\subsubsection{Product of equivalence relations \cite{MR1189891}*{1.3}}
\label{sss.prod-eqrel}
With the notation and conditions of \eqref{not.act.cor}.
Let $E$, $E'$ be adequate equivalent relations on algebraic cycles
on $SmProj_k$ \eqref{ss.ad.eq.rel}.  The product, $E\ast E'$, of
$E$ and $E'$ is the adequate equivalent relation defined
as follows:
$\alpha \in CH^n_{E\ast E'}(X)$ if there exist $Y\in SmProj_k$ 
and $\beta _i \in CH^{n_i}_E (Y\times X)$,
$\gamma _i \in CH^{m_i}_{E'} (Y\times X)$, $i=1,\ldots , r$;
$n=n_i + m_i -d_Y$,
such that $\alpha = \sum _i p_{X\ast}(\beta _i \cdot \gamma _i)$, 
\cite{MR1189891}*{Rmk. 1.6.1}.

Let $E=\mathrm{alg}$ denote the adequate equivalence relation
on algebraic cycles on $SmProj_k$ given by algebraic
equivalence \cite{MR116010}*{2.2}.  Then, by \cite{MR1189891}*{Ex. 1.9},
$\mathrm{alg}^{\ast l}$, $l\geq 2$, is Samuel's $l$-cubic
equivalence \cite{MR116010}*{2.4}.

\subsubsection{H. Saito's filtration \cite{MR1189891}*{Ex. 1.8}}
\label{sss.HSai-fil}
With the notation and conditions of \eqref{not.act.cor}.
Let $H$ be a Weil cohomology theory \cite{MR0292838}*{1.2} on
$SmProj_k$, with cycle class map $\mathrm{cl}_X: CH^n(X)
\rightarrow H^{2n}(X)$, $n\geq 0$.  
We will abuse notation and also write $H$
for the adequate equivalence relation given by 
\emph{homological equivalence} \cite{MR0292838}*{Prop. 1.2.3}.
Namely, $CH^n_H(X)\subseteq CH^n(X)$ is defined as
the kernel of the cycle class map $\mathrm{cl}_X$.

We consider the products $H^{\ast l}$, $l\geq 1$, of adequate equivalence
relations \eqref{sss.prod-eqrel}.  Thus, we obtain a descending
filtration for the Chow groups:
\begin{align}  \label{diag.HSai-fil}
\ldots \subseteq CH^n_{H^{\ast l+1}}(X) \subseteq
CH^n_{H^{\ast l}}(X) \subseteq \ldots \subseteq CH^n_{H}(X),
\end{align}
The filtration \eqref{diag.HSai-fil}
is of particular interest since it is a natural candidate for a
Bloch-Beilinson filtration \cite{MR1744947}*{Thm. 4.1}.

\subsubsection{Regular homomorphisms}	\label{def.reg.hom}
With the notation and conditions of \eqref{ss.nt.abvar}.  Furthermore,
we assume that the base field $k$ is algebraically closed.
Let $X\in SmProj_k$ of dimension $d$,
$A \in SmProj_k$ an
abelian variety and $1\leq n \leq d$. A group homomorphism
$\psi : CH^n_{\mathrm{alg}}(X) \rightarrow A(k)$ is
\emph{regular} \cite{MR116010}*{p. 477}, if for every $Y\in SmProj _k$,
$\Lambda \in CH^n(Y\times X)$, $y_0 \in Y(k)$;   there exists a map
$\psi _{\Lambda, y_0}: Y\rightarrow A$ in $Sm_k$, such that
for $y\in Y(k)$: $\psi _{\Lambda, y_0}(y)
=\psi(\Lambda ([y]-[y_0]))$ \eqref{not.act.cor}.

Let $\mathrm{alb}_X : CH^d_{\mathrm{alg}}(X)
\rightarrow \mathrm{Alb}(X) (k)$  be the Albanese map, which is a
regular homormorphism by the naturality of the Albanese map.

Consider a regular homomorphism
$\psi : CH^n_{\mathrm{alg}}(X) \rightarrow A(k)$,
where $A$ is an abelian variety of dimension $d_A$.  
We say that $\psi$ is \emph{cycle induced}
\cite{MR238857}*{p. 1194}, if there exist $\Lambda \in 
CH^{d+d_A-n} (X\times A)$,  and
an integer $r\geq 1$ such that for every $\alpha \in CH^n_{\mathrm{alg}}
(X)$: $r\psi(\alpha)=\mathrm{alb}_A (\Lambda (\alpha))\in A(k)$,
\eqref{not.act.cor}.

\section{The orthogonal filtration for the functor of points of an
abelian variety}  \label{s.mainres}

\subsection{}  \label{ss.results}
With the notation and conditions of \eqref{ss.nt.abvar}.  Furthermore,
we assume that the base field $k$ is algebraically closed. 
Let $X\in SmProj_k$ of dimension $d$, 
 $x_0\in X(k)$ and  $f:X \rightarrow A=\mathrm{Alb}(X)$ be
the canonical map into the Albanese variety such that $f(x_0)=0$.
This map is classified by a morphism in $\DM{k}$ \eqref{ss.nt.abvar}:
\begin{align}  \label{diag.mot-alb}
\mathcal A _f :M(X) \rightarrow \mathcal A 
\end{align}

Let $\Lambda \in CH^{d+n}(Y\times X)$, $Y\in SmProj_k$ of
dimension $d_Y$, $d_Y\geq n\geq 0$.
Consider $\Lambda : M(Y)\otimes M(X)
\rightarrow \sphere (d+n)[2d+2n]$ in $\DM{k}$, \eqref{sss.ident.chgps}.
Dualizing $M(X)$ \cite{MR1764202}*{Thm. 4.3.7},
\cite{MR2399083}*{Prop. 6.7.1 and \S 6.7.3}  we obtain the
following map in $\DM{k}$:
\begin{align}  \label{diag.corr-cyc-ind}
\Lambda _X : M(Y)(-n)[-2n] \rightarrow M(X)
\end{align}

Combining \eqref{diag.mot-alb}, \eqref{diag.corr-cyc-ind} and the
counit $\theta _{-1}$ \eqref{def.orthogonal.adj} in the orthogonal filtration,
we obtain the solid arrows in the following diagram in $\DM{k}$:
\begin{align}  \label{diag.mainthm}
\begin{split}
\xymatrix{ 
&& bc_{\leq -1}(\mathcal A) \ar[d]^-{\theta _{-1}^\mathcal A}\\
M(Y)(-n)[-2n] \ar[r]_-{\Lambda _X} \ar@{-->}[urr]^-{\Lambda '}
& M(X) \ar[r]_-{\mathcal A _f}& \mathcal A}
\end{split}
\end{align}

\begin{thm}  \label{thm.main}
With the notation and conditions of \eqref{ss.results} and
\eqref{diag.mot-alb}-\eqref{diag.mainthm}.
There exist a map $\Lambda ' : M(Y)(-n)[-2n]\rightarrow
bc _{\leq -1}(\mathcal A)$ in $\DM{k}$
such that \eqref{diag.mainthm} commutes
if and only if the following condition holds:

\subsubsection{}  \label{sss.main-cond}
For every $\gamma \in CH^{d_Y-n}(Y)$:
$\mathrm{alb}_X(\Lambda (\gamma )-\mathrm{deg} 
(\Lambda (\gamma ))[x_0])
=0 \in \mathrm{Alb}(X)(k)$.
\end{thm}
\begin{proof}
By \cite{MR3614974}*{5.3.2}, the existence of $\Lambda '$ making
\eqref{diag.mainthm} commute is equivalent to the following
composition being zero in $\DM{k}$ \eqref{ss.sf}:
\begin{align*}
\xymatrix{ f_0(M(Y)(-n)[-2n]) \ar[d]_-{\epsilon_0} \ar@{-->}[drr]^-0
&& \\
M(Y)(-n)[-2n] \ar[r]_-{\Lambda _X}
& M(X) \ar[r]_-{\mathcal A _f}& \mathcal A}
\end{align*}
Then the result follows from \eqref{prop.main-res}.
\end{proof}

\subsection{}  \label{ss.detailspf}
With the notation and conditions of \eqref{thm.main}.
We will write $\inthomeff$ for the internal $\Hom$-functor
in $\DMeff{k}$.

\subsubsection{}  \label{sss.higherCHgps}
We observe that for every $Z\in Sm_k$, $r\in Z$:
\begin{align*}
   \Hom_{\DMeff{k}}(M(Z)[r], \inthomeff (\sphere (n)[2n], M(Y)))\cong
   CH^{d-n}(Y\times Z,r)
\end{align*}
where the groups on the right are Bloch's higher Chow groups \cite{MR0852815}.  In fact, this follows by adjointness and
combining Poincar\'e duality \cite{MR1764202}*{Thm. 4.3.7},
\cite{MR2399083}*{Prop. 6.7.1 and \S 6.7.3} with \cite{MR1883180}.
Thus, we conclude that for $r<0$:
$\mathbf h _r(\inthomeff (\sphere (n)[2n], M(Y)))=0 \in \HINST{k}$
\eqref{ss.nt.htstr}.  Hence:
\begin{align}   \label{diag.inhomeff}
\inthomeff (\sphere (n)[2n], M(Y)) \in (\DMeff{k})_{\geq 0}.
\end{align}

\subsubsection{}
We claim that $f_0(M(Y)(-n)[-2n]) \in (\DMeff{k})_{\geq 0}$
\eqref{ss.nt.htstr}.
In effect, combining \cite{MR2600283}*{Lem. 5.9}, 
\cite{MR2249535}*{Prop. 1.1} with \cite{MR3614974}*{3.3.3(2)} we
deduce that:
\begin{align*}
f_0(M(Y)(-n)[-2n])\cong f_n(M(Y))(-n)[-2n]
\cong \inthomeff (\sphere (n)[2n], M(Y)),
\end{align*}
so, the assertion follows from \eqref{diag.inhomeff}.

Thus, we obtain the following distinguished triangle in $\DMeff{k}$
\eqref{ss.nt.htstr}:
\begin{align}  \label{diag.dtria.inthom}
\xymatrix{\tau _{\geq 1}E \ar[r]^-{t_1}& 
f_0(M(Y)(-n)[-2n])=E \cong \tau_{\geq 0}E \ar[r]^-{\sigma _0}
& \mathbf h _0 E .}
\end{align}
Now, recall that $\mathcal A$ is the functor of points of the Albanese
variety of $X$, so, $\mathcal A \in \HINST{k}$ \eqref{ss.nt.htstr},
and we conclude that
$\Hom _{\DMeff{k}}(\tau _{\geq 1}E, \mathcal A) =0= 
\Hom _{\DMeff{k}}(\tau _{\geq 1}E[1], \mathcal A)$, \eqref{diag.dtria.inthom}.

Therefore, there exists a unique map, $e$,
such that the following diagram in $\DM{k}$
commutes \eqref{diag.mainthm}, \eqref{diag.dtria.inthom}:
\begin{align}  \label{diag.evalinto-ab}
\begin{split}
\xymatrix{ f_0(M(Y)(-n)[-2n])=E \ar[d]_-{\epsilon _0} \ar[rr]^-{\sigma _0}
&& \mathbf h _0 E \ar[d]^-e \\
M(Y)(-n)[-2n] \ar[r]_-{\Lambda _X}
& M(X) \ar[r]_-{\mathcal A _f}& \mathcal A}
\end{split}
\end{align}
and we deduce that:
\subsubsection{}  \label{sss.int-step}
$\mathcal A _f \circ \Lambda _X \circ \epsilon _0 =0$
if and only if $e=0$.

Moreover,
on the one hand,
combining Lieberman's lemma \cite{MR1644323}*{Prop. 16.1.1} with
\cite{MR1764202}*{Prop. 2.1.4, Thm. 3.2.6},  we conclude that
$\Lambda _X$ \eqref{diag.evalinto-ab}
induces the following map of abelian groups:
\begin{align} \label{action.corr-dual}
\begin{split}
\xymatrix@R=0.1pc{
\Hom _{\DM{k}}(\sphere , M(Y)(-n)[-2n]) \ar[r]^-{\Lambda _{X \ast}}& 
\Hom_{\DM{k}}(\sphere , M(X))\\
\wr  ||& \wr ||\\
CH^{d_Y-n}(Y) & CH^d(X)\\
\gamma \ar@{|->}[r]&  
\Lambda (\gamma)}
\end{split}
\end{align}

On the other hand,
we claim that $\mathcal A _f$ \eqref{diag.evalinto-ab}
induces the following map of abelian groups \eqref{def.reg.hom},
\eqref{ss.results}:
\begin{align} \label{action.alb-map}
\begin{split}
\xymatrix@R=0.1pc{
\Hom _{\DM{k}}(\sphere , M(X)) \ar[r]^-{\mathcal A _{f \ast}}& 
\Hom_{\DM{k}}(\sphere , \mathcal A )\\
\wr  ||& \wr ||\\
CH^{d}(X) & \mathrm{Alb}(X)(k)\\
\alpha \ar@{|->}[r]&  
\mathrm{alb}_X (\alpha -\mathrm{deg}(\alpha)[x_0])}
\end{split}
\end{align}
In effect, by construction 
$\mathcal A _f \in \Hom_{\DMeff{k}}(M(X), \mathcal A)
\cong \Gamma (X, \mathcal A)$ \eqref{ss.nt.abvar} classifies the canonical map
$f:X\rightarrow \mathrm{Alb}(X)$ such that $f(x_0)=0$
\eqref{ss.results}, which implies that
$\mathcal A _{f\ast}$ \eqref{action.alb-map}  restricted to $CH^d_{\mathrm{alg}}(X)
\subseteq CH^d(X)$ is given by the Albanese map
$\mathrm{alb}_X:CH^d_{\mathrm{alg}}(X)\rightarrow \mathrm{Alb}(X)(k)$
and that $\mathcal A _{f \ast}(\mathbb Z [x_0])=0$.  Then, the claim
\eqref{action.alb-map} follows from the splitting
$CH^d(X)\cong CH^d_{\mathrm{alg}}(X)\oplus \mathbb Z$,
$\alpha \mapsto (\alpha -\mathrm{deg}(\alpha)[x_0], \mathrm{deg}(\alpha))$.

\subsubsection{}
Then,
combining \eqref{action.corr-dual} and \eqref{action.alb-map} 
we conclude that the composition
$\mathcal A _f \circ \Lambda _X$ \eqref{diag.evalinto-ab} 
induces the following map
of abelian groups:
\begin{align} \label{action.corr-alb}
\begin{split}
\xymatrix@R=0.1pc{
\Hom _{\DM{k}}(\sphere , M(Y)(-n)[-2n]) \ar[r]^-{(\mathcal A _f \circ
\Lambda _X )_{\ast}}& 
\Hom_{\DM{k}}(\sphere , \mathcal A)\\
\wr  ||& \wr ||\\
CH^{d_Y-n}(Y) & \mathrm{Alb}(X)(k)\\
\gamma \ar@{|->}[r]&  
\mathrm{alb}_X(\Lambda (\gamma )-\mathrm{deg} 
(\Lambda (\gamma ))[x_0])}
\end{split}
\end{align}

\begin{rmk}
Combining the bottom row of \eqref{diag.evalinto-ab} with
\eqref{action.corr-alb} we obtain 
a motivic presentation for cycle induced regular homomorphisms
\eqref{def.reg.hom}.
\end{rmk}

\begin{prop}  \label{prop.main-res}
With the notation and conditions of \eqref{thm.main},
\eqref{ss.detailspf} and \eqref{sss.higherCHgps}-\eqref{action.corr-alb}.
Then $\mathcal A _f \circ \Lambda _X \circ \epsilon _0 =0$ 
\eqref{diag.evalinto-ab}
if and only if the induced map $(\mathcal A _f \circ \Lambda _X)_\ast=0$
\eqref{action.corr-alb}.
\end{prop}
\begin{proof}
$(\Rightarrow)$:  By the universal property of the counit $\epsilon _0:
 f_0\rightarrow id$ \cite{MR3614974}*{3.3.1}, we deduce that the induced
 map 
\begin{align*}
\epsilon _{0 \ast}: \Hom_{\DM{k}}(\sphere ,f_0(M(Y)(-n)[-2n])) 
\overset{\cong}{\rightarrow}
\Hom_{\DM{k}}(\sphere , M(Y)(-n)[-2n])
\end{align*}
is an isomorphism,
since $\sphere \in \DMeff{k}$.  Thus, we conclude that
$(\mathcal A _f \circ \Lambda _X)_\ast=0$ \eqref{action.corr-alb}, since 
$\mathcal A _f \circ \Lambda _X \circ \epsilon _0 =0$.

$(\Leftarrow)$:  Since \eqref{diag.evalinto-ab} commutes and
$(\mathcal A _f \circ \Lambda _X)_\ast=0$ \eqref{action.corr-alb},
we deduce that the following composition of maps of abelian groups
is zero: 
\begin{align*}  
\xymatrix{\Hom_{\DM{k}}(\sphere, E)  \ar[rr]^-{\sigma _{0 \ast}}_-\cong
 \ar[drr]_-0
&& \Hom_{\DM{k}}(\sphere, \mathbf h _0 E) \ar[d]^-{e_\ast} \\
&& \Hom_{\DM{k}}(\sphere , \mathcal A)\cong \mathrm{Alb}(X)(k)}
\end{align*} 
where $E=f_0(M(Y)(-n)[-2n])$ \eqref{diag.evalinto-ab}.  Now, we
observe that
$\sigma _{0\ast}$ is an isomorphism, since \eqref{diag.dtria.inthom} is
a distinguished triangle in $\DMeff{k}$ and
$\spec{k}$ is a point for the Nisnevich topology \eqref{ss.nt.htstr}.
Thus, it follows that $e_{\ast}=0$, so, by combining
\cite{MR1764202}*{Prop. 3.1.9 and 3.2.3} with
\eqref{prop.kpts.sep} we deduce that $e=0$ \eqref{diag.evalinto-ab}.
Then the result follows from \eqref{sss.int-step}.
\end{proof}

\subsection{}  \label{ss.alb-ker}
With the notation and conditions of \eqref{ss.results}.
We will write $T(X)\subseteq CH^d_{\mathrm{alg}}(X)$ 
for the Albanese kernel of $X$, i.e. for the kernel of the Albanese map
$\mathrm{alb}_X: CH^d_{\mathrm{alg}}(X)\rightarrow \mathrm{Alb}(X)(k)$.

\begin{thm}  \label{thm.square-eq}
With the notation and conditions of \eqref{ss.results} and
\eqref{diag.mot-alb}-\eqref{diag.mainthm}.
Furthermore, we assume that $\Lambda \in CH^{d+n}_{\mathrm{alg}
^{\ast 2}}
(Y\times X)$ \eqref{sss.prod-eqrel}.
Then,
there exist a map $\Lambda ' : M(Y)(-n)[-2n]\rightarrow
bc _{\leq -1}(\mathcal A)$ in $\DM{k}$
such that \eqref{diag.mainthm} commutes.
\end{thm}
\begin{proof}
By \eqref{thm.main} it suffices to show that \eqref{sss.main-cond} holds.
Consider $\gamma \in CH^{d_Y-n}(Y)$.  Since 
$\Lambda \in CH^{d+n}_{\mathrm{alg}^{\ast 2}}
(Y\times X)$ we conclude that 
$\Lambda(\gamma) \in CH^d _{\mathrm{alg}^{\ast 2}}(X)\subseteq 
CH^d _{\mathrm{alg}}(X)$ \eqref{diag.HSai-fil}, so,
$\mathrm{deg}(\Lambda (\gamma))=0$.

	Hence, it is enough to see that $\Lambda (\gamma) \in T(X)$
\eqref{ss.alb-ker}, the Albanese kernel of $X$, which follows from
\cite{MR116010}*{p. 479, Prop. 11} since $\Lambda(\gamma) \in CH^d _{\mathrm{alg}^{\ast 2}}(X)$.
\end{proof}

If we restrict to the case $k=\mathbb C$, then
it is possible to strengthen \eqref{thm.square-eq}:

\begin{thm}  \label{thm.square-homeq}
With the notation and conditions of \eqref{ss.results} and
\eqref{diag.mot-alb}-\eqref{diag.mainthm}.
Furthermore, we assume that $\Lambda \in CH^{d+n}_{H^{\ast 2}}
(Y \times X)$ \eqref{sss.HSai-fil}-\eqref{diag.HSai-fil} and that
$k=\mathbb C$.
Then,
there exist a map $\Lambda ' : M(Y)(-n)[-2n]\rightarrow
bc _{\leq -1}(\mathcal A)$ in $\DM{k}$
such that \eqref{diag.mainthm} commutes.
\end{thm}
\begin{proof}
By \eqref{thm.main} it suffices to show that \eqref{sss.main-cond} holds.
Consider $\gamma \in CH^{d_Y-n}(Y)$.  Since 
$\Lambda \in CH^{d+n}_{H^{\ast 2}}
(Y\times X)$ we conclude that 
$\Lambda(\gamma) \in CH^d _{H^{\ast 2}}(X)\subseteq 
CH^d _{H}(X)$ \eqref{diag.HSai-fil}.  Now, we recall that for zero-cycles,
algebraic, homological and numerical equivalence coincide, so,
we deduce that
$\mathrm{deg}(\Lambda (\gamma))=0$ and that $\Lambda (\gamma)\in
CH^d_{\mathrm{alg}}(X)$.
	Hence, it is enough to show that $\Lambda (\gamma) \in T(X)$
\eqref{ss.alb-ker}, the Albanese kernel of $X$, which follows from
\cite{MR1189891}*{Rmk. 6.2} since $\Lambda (\gamma) \in
CH^d _{H^{\ast 2}}(X) \cap CH^d_{\mathrm{alg}}(X)$.
\end{proof}

\begin{rmk}
The condition $\Lambda \in CH^{d+n}_{\mathrm{alg}^{\ast 2}}
(Y\times X)$ in \eqref{thm.square-eq} is sharp in the following sense:
there exist $\Lambda \in CH^{d+n}_{\mathrm{alg}}(Y\times X)$ such that
\eqref{thm.square-eq} does not hold.  Namely, consider
$\alpha \in CH^d_{\mathrm{alg}}(X)$ such that $\mathrm{alb}_X(\alpha)
\in \mathrm{Alb}(X)(k)$ is a non-torsion element
and $\beta \in CH^n(Y)$  which is not numerically equivalent to zero, i.e.
there exists $\gamma _0 \in CH^{d_Y-n}(Y)$ such that
$\mathrm{deg}(\beta \cdot \gamma _0)\neq 0$.  We claim that
\eqref{thm.square-eq} does not hold for the exterior product
$\Lambda =\beta \otimes \alpha   \in CH^{d+n}_{\mathrm{alg}}(Y\times X)$.

In effect, by the projection formula:
$\Lambda (\gamma _0) =\mathrm{deg}(\beta \cdot \gamma_0)\alpha$, so,
$\mathrm{deg}(\Lambda (\gamma _0))=\mathrm{deg}(\beta \cdot \gamma_0)\cdot \mathrm{deg}(\alpha)=0$,
since
$\alpha \in CH^d_{\mathrm{alg}}(X)$.
Then, $\mathrm{alb}_X(\Lambda (\gamma _0) -\mathrm{deg}(\Lambda
(\gamma _0))[x_0])= \mathrm{alb}_X(\Lambda (\gamma _0))=
\mathrm{deg}(\beta \cdot \gamma_0)\cdot \mathrm{alb}_X(\alpha)$,
which is a non-zero element in $\mathrm{Alb}(X)(k)$ by the choice
of $\alpha$.  Thus, the
 condition \eqref{sss.main-cond} does not hold, so,
\eqref{thm.square-eq}
does not hold as well by \eqref{thm.main}.
\end{rmk}

%%%%%%%%%%%%%%%%%%%%%%%%%%%%%%%%%%%%%%%%%%%%
%%%%%%%%%%%%%%%%%%%%%%%%%%%%%%%%%%%%%%%%%%%%
%%%%%%%%%%%%%%%%%%%%%%%%%%%%%%%%%%%%%%%%%%%%

\bibliography{biblio_chowcor}
\bibliographystyle{abbrv}

\end{document}